\documentclass[12pt]{amsart}
\usepackage{amsthm,amsmath,a4,amssymb,verbatim,url}

\newtheorem{thm}{Theorem}[section]
\newtheorem{cor}[thm]{Corollary}
\newtheorem{lem}[thm]{Lemma}
\newtheorem{prop}[thm]{Proposition}
\theoremstyle{definition}

\theoremstyle{remark}
\newtheorem{rem}[thm]{Remark}
\newtheorem{ex}[thm]{Example}
\numberwithin{equation}{section}

\newcommand{\GL}{\textnormal{GL}}

\newcommand{\Z}{\mathbf{Z}}

\newcommand{\Q}{\mathbf{Q}}

\begin{document}

\author{Khalid Bou-Rabee}
\address{The City College of New York, 160 Convent Ave, New York, NY, 10031 USA}
\email{khalid.math@gmail.com}

\author{Yves Cornulier}
\address{CNRS -- D\'epartement de Math\'ematiques, Universit\'e Paris-Sud, 91405 Orsay, France}

\email{yves.cornulier@math.u-psud.fr}

\title{Systolic growth of linear groups}
\date{August 21, 2014}

\thanks{K.B.~is supported in part by NSF DMS-1405609; Y.C.~is supported in part by ANR GSG 12-BS01-0003-01}



\begin{abstract}
We prove that the residual girth of any finitely generated linear group is at most exponential.
This means that the smallest finite quotient in which the $n$-ball injects has at most exponential size. 
If the group is also not virtually nilpotent, it follows that the residual girth is precisely exponential.
\end{abstract}
\maketitle

\section{Introduction}

Let $\Gamma$ be a group with a finite generating subset $S$, and $|\cdot|_S$ the corresponding word length. We assume for convenience that $S$ is symmetric and contains the unit, so that $S^n$ is equal to the $n$-ball. The following three functions are attached to $(\Gamma,S)$:
\begin{itemize}
\item the growth: the cardinal $b_{\Gamma,S}(n)$ of $S^n$;
\item the systolic growth: the function $\sigma_{\Gamma,S}$ mapping $n$ to the smallest $k$ such that some subgroup $H$ of index $k$ contains no nontrivial element of the $n$-ball; if no such $k$ exists, we define it as $+\infty$;
\item the residual girth, or normal systolic growth $\sigma'_{\Gamma,S}$: same definition, with the additional requirement that $H$ is normal.
\end{itemize}

The growth is always defined and is at most exponential, while the systolic growth and residual girth take finite values if and only if $\Gamma$ is residually finite, and in this case they can be larger than exponential, as the example in \cite{BSe} show. Furthermore, we have the obvious inequalities \[b_{\Gamma,S}(n)\le\sigma_{\Gamma,S}(2n+1)\le\sigma^\triangleleft_{\Gamma,S}(2n+1).\] The asymptotic behavior of these functions, for finitely generated groups, does not depend on the finite generating subset.
  
A simple example for the residual girth grows strictly faster than the systolic growth is the case of the integral Heisenberg group, for which the growth and systolic growth behaves as $n^4$ while the residual girth grows as $n^6$ (see \cite{BS,C}). Also the systolic growth may grow faster than the growth and actually can grow arbitrarily fast. We show here that in linear groups, this is not the case.

\begin{thm}\label{main}
Assume that $\Gamma$ admits a faithful finite-dimensional representation over a field (or a product of fields).
Then the residual girth (and hence the systolic growth) of $\Gamma$ are at most exponential. In particular, if $\Gamma$ is not virtually nilpotent, then its residual girth and its systolic growth are exponential.
\end{thm}

Such a result was asserted by Gromov \cite[p.334]{Gro} for subgroups of $\mathrm{SL}_d(\Z)$, under some technical superfluous additional assumption (non-existence of nontrivial unipotent elements).

The proof of Theorem \ref{main} consists in finding small enough quotient fields of the ring of entries, while ensuring that the $n$-ball is mapped injectively. The argument can be simplified in case $\Gamma\subset\GL_d(\Q)$, since then reduction modulo $p$ for all $p$ large enough work with no further effort; in this case the finite quotients are explicit, while in the general case we only find a suitable quotient field using a counting argument.

\begin{ex}
The group $\Z\wr\Z$ has an exponential residual girth. Another example is $(\Z/6\Z)\wr\Z$, which is linear over a product of 2 fields, but not over a single field.
\end{ex}

\begin{rem}
Closely related functions are the residual finiteness growth, which maps $n$ to the smallest number $s_{\Gamma,S}(n)$ such that for every $g\in S^n\smallsetminus\{1\}$, there is a finite index subgroup of $\Gamma$ avoiding $g$, and $s_{\Gamma,S}^\triangleleft(n)$ defined in the same way with only normal finite index subgroups. For finitely generated group that are linear over a field, a polynomial upper bound for these functions is established in \cite{BM}, and in the case of higher rank arithmetic groups, the precise behavior is obtained in \cite{BK}: for instance, for $\mathrm{SL}_d(\Z)$ for $d\ge 3$, the normal residual finiteness growth grows as $n^{d^2-1}$.  
\end{rem}

\section{Preliminaries on polynomials over finite fields}

\begin{lem}\label{nmp}
Let $F$ be a finite field with $q$ elements. Given an integer $n\ge 1$, the number of irreducible monic polynomials of degree $n$ in $F[t]$ is $\le q^n/n$ and $\ge (q^n-q^{n-1})/n$.

\end{lem}
\begin{proof}
The case $n=1$ being trivial, we can assume $n\ge 1$.
By Gauss' formula this number $N_q(n)$ is equal to $(1/n)\sum_{d|n}\mu(n/d)q^d$, where $\mu$ is M\"obius' function. Let $p>1$ be the smallest prime divisor of $n$. Then 
\[\sum_{d|n}\mu(n/d)q^d=q^n-q^{n/p}+\sum_{d|n,d>p}\mu(n/d)q^d\le q^n-q^{n/p}+\sum_{d|n,d>p}q^d\]
\[\le q^n-q^{n/p}+\sum_{k=0}^{n/p-1}q^k\le q^n\]

A similar argument shows that $nN_q(n)\ge q^n-q^{1+n/p}$, which is $\ge q^n-q^{n-1}$ if $n\ge 3$; the cases $n\le 2$ being trivial. 
\end{proof}




\begin{lem}\label{surv}
Let $F$ be a field with $q$ elements. Let $P\in F[t]$ be a nonzero polynomial of degree $\le n$. Then $P$ survives in a quotient field of $F[t]$ of cardinal $\le 2nq$.

\end{lem}

\begin{proof}
Let $m\ge 1$ be the largest number such that every irreducible polynomial of degree $m-1$ divides $P$. Let us check that $q^m\le 2nq$; the case $m=1$ being trivial, we assume $m\ge 2$. By Lemma \ref{nmp}, there are $\ge (q^{m-1}-q^{m-2})/(m-1)$ monic irreducible polynomials of degree $m-1$. Hence their product, which has degree $\ge q^{m-1}-q^{m-2}$, divides $P$. Thus $q^{m-1}-q^{m-2}\le n$.
We have $1-q^{-1}\ge 1/2$; thus $\frac12q^mq^{-1}\le n$, that is $q^{m}\le 2nq$.


Some irreducible polynomial of degree $m$ does not divide $P$, hence the quotient provides a field quotient of cardinal $q^m\le 2nq$ in which $P$ survives. 
\end{proof}


\begin{cor}\label{corqfc}
Let $F$ be a field with $q$ elements and $P$ a nonzero polynomial in $F[t_1,\dots,t_k]$, of degree $\le n$ with respect to each indeterminate. Then $P$ survives in a quotient field of cardinal $\le (2n)^kq$.
\end{cor}
\begin{proof}
Induction on $k$. The result is trivial for $k=0$. Write \[P=\sum_{i=0}^nP_i(t_1,\dots,t_{k-1})t_k^i.\] Some $P_i$ is nonzero; fix such $i$. Then there exists, by induction, some quotient field $L$ of $F[t_1,\dots,t_{k-1}]$ of cardinal $\le (2n)^{k-1}q$ in which $P_i$ survives. Then the image of $P$ in $L[t_k]$ has degree $\le n$ and is nonzero; hence by Lemma \ref{surv}, it survives in a quotient field of cardinal $2n((2n)^{k-1}q)=(2n)^kq$.
\end{proof}


\section{Conclusion of the proof}

\begin{prop}\label{rgp}
Every finitely generated group that is linear over a field of characteristic $p$ has at most exponential residual girth.
\end{prop}
\begin{proof}
Such a group embeds into $\GL_d(K)$ where $K$ is an extension of degree $b$ of some field $K'=F_q(t_1,\dots,t_k)$, and hence embeds into $\GL_{bd}(K')$. Hence it is no restriction to assume that the group is contained in $\GL_d(F_q(t_1,\dots,t_k))$. We let $S$ be a finite symmetric generating subset with 1; it is actually contained in $\GL_d(F_q[t_1,\dots,t_k][Q^{-1}])$ for some nonzero polynomial $Q$.

Write $S=Q^{-\lambda}T$ with $\lambda$ a non-negative integer and $T\subset\mathrm{Mat}_d(F_q[t_1,\dots,t_k])$; write $s=\#(S)=\#(T)$. If $x$ is a matrix, let $b(x)$ be the product of all its nonzero entries (thus $b(0)=1$).
Let $m$ be such that every entry of every element of $T$ has degree $\le m$ with respect to each variable. Then in $T^{2n}$, every entry of every element has degree $\le 2nm$ with respect to each variable. Define $x_n=\prod_{y\in T^{2n}}b(y-1)$. Thus $x_n$ is a product of at most $d^2s^{2n}$ polynomials of degree $\le 2nm$ with respect to each variable. Define $x'_n=x_nQ$; assume that $Q$ has degree $\le\delta$ with respect to each variable, so that $x'_n$ has degree $\le 2d^2mns^{2n}+\delta$ with respect to each variable.


Then, by Corollary \ref{corqfc}, $x'_n$ survives in a finite field $F_n$ of cardinal $q_1\le q(4d^2mns^{2n}+2\delta)^k$. Thus $S^n$ is mapped injectively into $\GL_d(F_n)$, which has cardinal \[\le q_1^{d^2}\le q^{d^2}(4d^2mns^{2n}+2\delta)^{kd^2}.\] Since $m,d,k,s,q$ are fixed, this grows at most exponentially with respect to $n$.
\end{proof}

\begin{prop}\label{rg0}
Every finitely generated group that is linear over a field of characteristic $0$ has at most exponential residual girth.
\end{prop}
\begin{proof}
Similarly as in the proof of Proposition \ref{rgp}, we can suppose that the group is contained in $\GL_d(\Q(t_1,\dots,t_k))$. We let $S$ be a finite symmetric generating subset with 1; it is actually contained in $\GL_d(\Z[t_1,\dots,t_k][r^{-1}Q^{-1}])$ for some nonzero integer $r\ge 1$ and nonzero polynomial $Q$ with coprime coefficients.

Write $S=(Qr)^{-\lambda}T$ with $\lambda$ a non-negative integer and $T\subset\mathrm{Mat}_d(\Z[t_1,\dots,t_k])$; write $s=\#(S)=\#(T)$. Let $R$ be an upper bound on coefficients of entries of elements of $T$, and let $M$ be an upper bound on the number of nonzero coefficients of entries of elements of $T$. Then any product of $2n$ elements of $T$ is a sum of $\le M^{2n}$ monomials, each with a coefficient of absolute value $\le R^{2n}$. Since any entry of an element in $T^{2n}$ is a sum of at most $d^{2n-1}$ such products, we deduce that the coefficients of entries of elements of $T^{2n}$ are $\le d^{2n-1}R^{2n}M^{2n}$. There exists a prime $p_n\in [2d^{2n-1}(RM)^2n,4d^{2n-1}(RM)^{2n}]$. There exists $n_0$ such that for every $n\ge n_0$, $2d^{2n-1}(RM)^2n$ is greater than any prime divisor of $r$, and $2d^{2n-1}(RM)^2n$ is greater than the lowest absolute value of a nonzero coefficient of $Q$. Now we always assume $n\ge n_0$. Then $S^{2n}$ is mapped injectively into $\GL_d((\Z/p_n\Z)[t_1,\dots,t_k][Q^{-1}])$.


Let $m$ be such that every entry of any element of $T$ has degree $\le m$ with respect to each variable. The previous proof provides a quotient $\GL_d(F_n)$ of $\GL_d((\Z/p_n\Z)[t_1,\dots,t_k][Q^{-1}])$ in which $S^n$ is mapped injectively, such that $\GL_d(F_n)$ has cardinal 
\[\le {p_n}^{d^2}(4d^2mns^{2n}+2\delta)^{kd^2}\]
Here $m,d,s,k$ are independent of $n$. The latter number is 
\[\le {(4d^{-1}(dRM)^{2n})}^{d^2}(4d^2mns^{2n}+2\delta)^{kd^2},\]
which grows at most exponentially with respect to $n$.
\end{proof}

\begin{proof}[Proof of Theorem \ref{main}]
First assume that $\Gamma$ is linear over some field.
By Propositions \ref{rgp} and \ref{rg0}, the residual girth, and hence the systolic growth, is at most exponential. If $\Gamma$ is not virtually nilpotent, then by the Tits-Rosenblatt alternative, it contains a free subsemigroup on 2 generators and hence has exponential growth, and therefore has at least exponential systolic growth and residual girth.

Now assume that $\Gamma$ is linear over some product of fields. Let $A$ be the ring generated by entries of $\Gamma$. This is a finitely generated reduced commutative ring; hence it has finitely many minimal prime ideals, whose intersection equals the set of nilpotent elements and hence is reduced to zero. Therefore $\Gamma$ embeds into a finite product of matrix group over various fields. We conclude that $\Gamma$ has at most exponential residual girth, using the following two general facts:
\begin{itemize}
\item suppose that $\Gamma_1,\dots,\Gamma_k$ are finitely generated groups and $\Gamma_i$ has residual girth asymptotically bounded above by some function $u_i\ge 1$, then the residual girth of $\prod_{i=1}^k\Gamma_i$ is asymptotically bounded above by $\prod u_i$;
\item if $\Lambda_1\subset\Lambda_2$ are finitely generated groups then the residual girth of $\Lambda_1$ is asymptotically bounded above by that of $\Lambda_2$.
\end{itemize}
\end{proof}




\end{document}